\documentclass[11pt,a4paper]{article}

\usepackage{graphicx}
\usepackage{amsmath}
\usepackage{amsfonts}
\usepackage{amssymb}
\usepackage{enumerate}
\usepackage{lscape}
\usepackage{longtable}
\usepackage{rotating}
\usepackage{multirow}
\usepackage{color}
\usepackage{url}
\usepackage{subfigure}
\usepackage{rotating}
\newtheorem{theorem}{Theorem}[section]

\usepackage[ruled,vlined,noline,linesnumbered]{algorithm2e}

\newtheorem{example}{Example}[section]

\newtheorem{lemma}{Lemma}[section]

\newtheorem{proposition}{Proposition}[section]

\newtheorem{assumption}{Assumption}[section]

\newenvironment{proof}[1][Proof]{\textbf{#1.} }{\ \rule{0.5em}{0.5em} \vspace{1ex}}

\def\real{\mathbb{R}}

\DeclareMathOperator{\argmin}{argmin}

\DeclareMathOperator{\dist}{dist}
\title{Convergence and Complexity Analysis of a Levenberg-Marquardt Algorithm for Inverse Problems
}%\author{E. Bergou\footnotemark[1] , Y. Diouane\footnotemark[2] , V. Kungurtsev\footnotemark[3]}
\author{
E. Bergou \thanks{MaIAGE, INRAE, Universit\'e Paris-Saclay, 78350 Jouy-en-Josas, France
 ({\tt elhoucine.bergou@inra.fr}).
}
\and
Y. Diouane\thanks{ISAE-SUPAERO, Universit\'e de Toulouse, 31055 Toulouse Cedex 4, France
 ({\tt youssef.diouane@isae-supaero.fr}).
}
\and
V. Kungurtsev \thanks{Department of Computer Science, Faculty of Electrical Engineering, Czech Technical University in Prague ({\tt vyacheslav.kungurtsev@fel.cvut.cz}).}
}
\begin{document}
\maketitle

%\renewcommand{\thefootnote}{\fnsymbol{footnote}}
%
%\footnotetext[1]{MaIAGE, INRA, Universit\'e Paris-Saclay, 78350 Jouy-en-Josas, France. {\tt elhoucine.bergou@inra.fr}}
% \footnotetext[2]{Institut Sup\'erieur de l'A\'eronautique et de l'Espace (ISAE-SUPAERO), Universit\'e de Toulouse, 31055 Toulouse Cedex 4, France. {\tt youssef.diouane@isae.fr}}
%\footnotetext[3]{Department of Computer Science, Faculty of Electrical Engineering, Czech
%Technical University in Prague. Research supported by the Czech Science Foundation project 17-26999S. {\tt vyacheslav.kungurtsev@fel.cvut.cz}}
\footnotesep=0.4cm
{\small
\begin{abstract}
The Levenberg-Marquardt algorithm is one of the most popular algorithms for finding the solution of nonlinear least squares problems. 
Across different modified variations of the basic procedure, the algorithm enjoys global convergence, a competitive worst case iteration complexity
rate, and a guaranteed rate of local convergence for both zero and nonzero small residual problems, under suitable assumptions. We introduce a novel Levenberg-Marquardt method
that matches, simultaneously, the state of the art in all of these convergence properties with a single seamless algorithm. 
Numerical experiments confirm the theoretical behavior of our proposed algorithm.
\end{abstract}

\begin{center}
\end{center}
}
%\begin{AMS}
%\end{AMS}

\section{Introduction}\label{sec:intro}
In this paper, we consider solving general nonlinear least squares problems where one may or may not have
a solution with a zero residual.
  Problems of this
nature arise in several important practical contexts, including inverse problems for ill-posed nonlinear 
continuous systems~\cite{ATarantola_2005} with applications
such as data assimilation~\cite{YTremolet_2007}.
 Usually the resulting 
least squares problems do not necessarily have a zero residual at any point, although the minimum residual may be small.

Recall that the Gauss-Newton method is an iterative procedure for solving nonlinear least squares problems by iteratively solving 
a linearized least squares subproblem. This subproblem may not be well-posed in the case of rank deficiency of the residual Jacobian function.
Furthermore, the Gauss-Newton method may not be globally convergent. The Levenberg-Marquardt (LM) method~\cite{KLevenberg_1944,DMarquardt_1963,MROsborne_1976} was developed to
deal with the rank deficiency of the Jacobian matrix   
and also to provide a globalization strategy for Gauss-Newton.
In this paper, we will present and analyze the global and local convergence results of a novel LM method for solving general nonlinear least squares problems, that carefully balances the opposing objectives of ensuring global convergence and stabilizing a fast local convergence regime.

In general, the goals of encouraging global and local convergence compete against each other. Namely, the regularization
parameter appearing in the subproblem should be allowed to become arbitrarily large in order to encourage global convergence,
by ensuring the local accuracy of the linearized subproblem, but the parameter must approach zero in order to function
as a stabilizing regularization that encourages fast local convergence. In the original presentation of the LM method 
in~\cite{KLevenberg_1944,DMarquardt_1963}, the regularization parameter is not permitted
to go to zero, and only global convergence is considered.

The strongest results for local convergence of LM are given in a series of papers beginning with~\cite{yamashita2001rate}
(followed by, e.g.,~\cite{fan2005quadratic} and~\cite{dan2002convergence}; see also~\cite{facchinei2013family}), 
wherein it is assumed that the residual function is zero at the solution. For the global convergence, the algorithm considered is a two-phase one, where
quadratic decline in the residual is tested with each step that is otherwise globalized by a line-search procedure. 

In the case of nonzero residuals, it has been found that a LM method converges locally at a linear rate, 
if the norm of the residual is sufficiently small and the 
regularization parameter goes to zero \cite{ipsen2011rank}.  
Our proof of linear convergence is simpler than in~\cite{ipsen2011rank}.
Worst-case iteration complexity bounds for LM methods applied to nonlinear least squares problems 
can be found in~\cite{KUeda_NYamashita_2010,KUeda_NYamashita_2012,RZhao_JFan_2016}. We show that our proposed algorithm has a complexity bound that matches these results, up to a logarithmic factor.

In this paper, we propose a method that successfully balances the multiple objectives in theoretical convergence properties, including (a) global convergence
for exact and inexact solutions of the subproblem, (b) worst case iteration complexity, and (c) local convergence for both zero and nonzero residual
problems. Table~\ref{table:lit} summarizes the literature on this class of methods; our proposed algorithm uniquely matches the state of the
art in all of these properties. 
\begin{table}[!h]
\centering
\caption{Convergence Properties of Levenberg-Marquardt Algorithms.}
\label{table:lit}
\begin{tabular}{|l|c|c|c|c|c|}
\hline
& \cite{KLevenberg_1944,DMarquardt_1963,MROsborne_1976} & \cite{KUeda_NYamashita_2010,KUeda_NYamashita_2012,RZhao_JFan_2016} 
& \cite{yamashita2001rate,fan2005quadratic,facchinei2013family} & \cite{ipsen2011rank} & \textbf{This work} \\ \hline
Global convergence & Yes & Yes & Two-phase or No & Yes & Yes \\ \hline
Complexity analysis& No & Yes & No & No & Yes \\ \hline
Local zero residual & No & No & Quadratic & Superlinear & Quadratic \\ \hline
Local nonzero  & No & No & No & Linear & Linear \\ 
 small residual &  &  &  & & \\ \hline
\end{tabular}
\end{table}

The outline of this paper is as follows. In Section \ref{sec:alg}, we present the proposed LM algorithm and address the inexact solution of the
linearized least squares subproblem. 
Section~\ref{sec:wcc} contains a worst-case complexity and global convergence
analysis of the proposed method. In Section \ref{sec:local},  we derive the local convergence theory. In Section \ref{sec:num}, preliminary numerical experiments are presented to
demonstrate the behavior of our algorithm. Finally, we conclude in Section \ref{sec:conclusion}.

\section{A Novel Levenberg-Marquardt Algorithm}\label{sec:alg}
In this paper, we consider the following nonlinear least squares problem
\begin{equation} \label{function-f}
 \min_{x \in \real^n} \; f(x) := \frac{1}{2}\|F(x)\|^2,
\end{equation}
where $F:\real^n \rightarrow \real^m$ is a (deterministic) vector-valued function,
assumed twice continuously differentiable. Here and in the rest of the text, $\| \cdot \|$ denotes the vector or matrix $l_2$-norm.
At each iteration $j$, the LM method computes (approximately) a step $s_j$ of the form
$- ( J_j^\top J_j + \gamma_j I )^{-1} J_j^\top F_j$,
corresponding to the unique solution of
\begin{equation}
\label{eq:LMsubproblem}
\min_{s \in \real^n} \; m_j(s) := \frac{1}{2}\|F_j + J_j s\|^2 + \frac{1}{2}\gamma_j \|s\|^2,
\end{equation}
where $\gamma_j>0$ is an appropriately chosen regularization parameter, $F_j:=F(x_j)$ and $J_j :=J(x_j)$ denotes the Jacobian of $F$ at $x_j$.

In deciding whether to accept a step $s_j$ generated by the subproblem~\eqref{eq:LMsubproblem}, the LM method can be seen as precursor of the trust-region method~\cite{Conn_Gould_Toin_2000}. In fact, it seeks to determine when the Gauss-Newton step is applicable or when it should be replaced by a slower but safer steepest descent step. 
 One considers the ratio $\rho_j$ between the actual reduction  $f(x_{j}) - f(x_j+s_j)$ attained in the
objective function and the reduction $m_j(0) - m_j(s_j)$ predicted by the
model.
Then, if $\rho_j$ is sufficiently greater than zero, the step is accepted
and~$\gamma_j$ is possibly decreased. Otherwise the step is
rejected and $\gamma_j$ is increased.

In this paper, we use the standard choice of the  regularization parameter
$$\gamma_j := \mu_j \|F(x_j)\|^2,$$ 
where $\mu_j$ is updated according to the ratio $\rho_j$.  
The considered LM algorithm using the above update strategy, as described in Algorithm \ref{alg:LM}, will be shown to be globally convergent with a  complexity bound of order $\epsilon^{-2}$  
 and have strong local convergence properties.
 
\noindent
\begin{algorithm}[ht]
\DontPrintSemicolon
\SetAlgoNlRelativeSize{0}
\caption{\bf A Levenberg-Marquardt algorithm.}
\label{alg:LM}
\begin{rm}
\begin{description}
%\item[]
\item[Initialization] \ \\
Choose the constants $\eta \in ]0,1[,~\mu_{\min}>0$ and $\lambda>1$.
Select $x_0$ and $\mu_0 \geq \mu_{\min}$. Set $\gamma_0 = \mu_0 \|F(x_0)\|^2$ and $\bar \mu =\mu_0$. 
%\vspace{1ex}
\item[For $j=0,1,2,\ldots$] \ \\
%\vspace{-2ex}
\begin{enumerate}
\item Solve (or approximately solve) (\ref{eq:LMsubproblem}), and let $s_j$ denote such a solution.
%\vspace{1ex}
\item Compute $\rho_j := \frac{f(x_{j}) - f(x_j+s_j)}{m_j(0) - m_j(s_j)}$.
%\vspace{1ex}
\item If $\rho_j \ge \eta$, then set $x_{j+1}=x_j+s_j$ and
 $ \mu_{j+1} \in [\max\{\mu_{\min},\bar\mu/\lambda\}, \bar\mu]$ and $\bar \mu = \mu_{j}$.

Otherwise, set $x_{j+1}=x_j$ and $\mu_{j+1}=\lambda\mu_j.$
%\vspace{1ex}
\item Compute $\gamma_{j+1} = \mu_{j+1}\|F(x_{j+1})\|^2$.
\end{enumerate}
\end{description}
\end{rm}
\end{algorithm}

This Algorithm has one particularly novel feature among LM methods: 
we have an auxiliary parameter
$\bar \mu$ which represents the last parameter corresponding to a successful step, introduced to balance the requirements
of global and local convergence. If the model is inaccurate, then $\mu_j$ is driven higher,
however, when we reach a region associated with the local convergence regime, the 
residual $\|F(x_j)\|$ should ultimately dominate the behavior
of $\gamma_j$ for successful steps. 
Step~1 of Algorithm~\ref{alg:LM} requires the approximate solution of
subproblem (\ref{eq:LMsubproblem}). As in trust-region methods,
there are different techniques to approximate the solution of this subproblem
that yield a globally convergent step.
For that it suffices to compute
a step $s_j$ that provides a reduction in the model at least as good as the one
produced by the so-called Cauchy step (defined as the minimizer of the model
along the negative gradient) which is given by
$$s_j^{\mbox{c}}\; := \; -\frac{\|\nabla f(x_j)\|^2}{\nabla f(x_j)^\top(J_{j}^\top J_{j}+\gamma_j I) \nabla f(x_j)} \nabla f(x_j).$$
The Cauchy step is cheap to calculate as it does not require any system solve.
Moreover, the LM method will be globally convergent if it uses
a step that attains a reduction in the model as good as a multiple of the Cauchy decrease.
Thus we will impose the following assumption on the step calculation:
\begin{assumption}
\label{ass:cauchy_decrease}
There exists $\theta_{fcd}>0$ such that for every iteration $j$:
$$m_j(0) - m_j(s_j)
\; \geq \; \frac{\theta_{fcd}}{2}\frac{\|\nabla f(x_j)\|^2}{\|J_{j}\|^2+\gamma_j}.$$ 
\end{assumption}
Despite providing a sufficient reduction in the model and being cheap to compute,
the Cauchy step is scaled steepest descent. In practice, a version of Algorithm~\ref{alg:LM} based solely on the  Cauchy
step would suffer from the same drawbacks as the steepest descent algorithm on ill-conditioned problems.

Since the Cauchy step is the first step of the
conjugate gradient method (CG) when applied to the minimization of
the quadratic $s \rightarrow m_j(s)$, it is natural to propose running CG
further and stopping only when the residual becomes sufficiently small. 
Since the CG generates iterates by minimizing the quadratic model over nested Krylov subspaces,
and the first subspace is the one generated by $\nabla f(x_j)$
(see, e.g.,~\cite[Theorem~5.2]{JNocedal_SJWright_2006}),
the decrease obtained at the first CG iteration (i.e., by the Cauchy step)
is at least attained by the remaining iterations. Thus Assumption~\ref{ass:cauchy_decrease} holds for all the iterates $ s_j^{\mbox{cg}}$ generated by the truncated-CG whenever it is initialized by the null vector.  The following lemma is similar to~\cite[Lemma 5.1]{EBergou_SGratton_LNVicente_2016} and will be useful  for our worst-case complexity analysis.
\begin{lemma} \label{lem:born_s}
For the three steps proposed (exact, Cauchy, and truncated-CG), one has that
\begin{eqnarray*}
\|s_j\| \le  \frac{\|\nabla f(x_j)\|}{\gamma_j} & ~~\mbox{and}~~ & | s_j^\top ( \gamma_j s_j + \nabla f(x_j) ) |  \le 
\frac{ \| J_{j}\|^2 \|\nabla f(x_j)\|^2 }{\gamma_j^{2}}.
\end{eqnarray*}
\end{lemma}
In what comes next, we will call  all iterations $j$ for which $\rho_j\ge \eta$ \emph{successful}, and we denote the set of their indices by the symbol $\mathcal{S}$, i.e., $\mathcal{S} := \{j \in \mathbb{N}  | ~~\rho_j  \ge \eta \}.$

%%%%%%%%%%%%%%%%%%%%%%%%%%%
\section{Worst-Case Iteration Complexity and Global Convergence}
\label{sec:wcc}
%%%%%%%%%%%%%%%%%%%%%%%%%%%
We now establish a worst-case complexity bound of Algorithm~\ref{alg:LM}. Namely, given a tolerance 
$\epsilon \in ]0,1[$, we aim at deriving the number of iterations, in the worst case, needed to reach an iterate 
$x_j$ such that
	\begin{equation}
	\label{eq:wcc:approxopt}
		\left\|\nabla f(x_j) \right\| < \epsilon \text{ or } \left\|F(x_j)\right\| < \max\left\{\epsilon,
(1+\epsilon)\| F(\bar x_j)\|\right\} 
	\end{equation}
where $\bar x_j \in \argmin_{\{x \in \real^n: \nabla f(x)=0\}} \|x_j - x \|$. 
Without loss of generality, we will assume that $\|F(\bar  x_j)\|$ is unique and independent from $x_j$, we will denote it by $\bar f$. It can be seen that, if we drop this assumption, then the same arguments in this section show asymptotic global convergence. Then, $\bar f := \sqrt{2f(\bar{x}_j)}=\|F(\bar{x}_j)\|$ can just be taken to be the value at the limit point of the sequence.
We start now by giving some classical assumptions.
\begin{assumption} \label{ass:f}
The function $f$ is continuously differentiable in an open set containing
$L(x_0) := \{ x \in \real^n: f(x) \leq f(x_0) \}$
with Lipschitz continuous gradient on $L(x_0)$ with the constant $\nu > 0$.
\end{assumption}
\begin{assumption} \label{ass:J} 
The Jacobian $J$ of $F$ is uniformly bounded, i.e., there exists $\kappa_{J} > 0$
such that $\|J \| \leq \kappa_{J}$ for all~$x$.
\end{assumption}	
We begin by obtaining a condition on the parameter $\mu_j$ that is sufficient for an iteration to be successful.
We omit the proof as it is standard, see for instance Lemma 5.2 in   \cite{EBergou_SGratton_LNVicente_2016}.
\begin{lemma}
\label{lemma:wcc:condmusuccess}
	Let Assumptions~\ref{ass:cauchy_decrease}, \ref{ass:f}, and \ref{ass:J} hold. 
	Suppose that at the $j$-th iteration of Algorithm~\ref{alg:LM}, one has
	\begin{equation}
	\label{eq:wcc:condmusuccess}
		\mu_j \; > \; \frac{\kappa}{\|F(x_j)\|^2}
	\end{equation}
	where
	\[
		\kappa :=  \frac{a+\sqrt{a^2+4a \kappa_J^2(1-\eta)}}{2(1-\eta)} \quad \mbox{and} \quad
		a := \frac{\tfrac{\nu}{2}+2\kappa_J^2}{\theta_{fcd}}.
	\]
	Then, the iteration is successful.
\end{lemma}
Our next result states that, when the gradient norm stays bounded away from zero, the parameter
$\mu_j$ cannot grow indefinitely. Without loss of generality, we assume that $\epsilon \le  \sqrt{\frac{\lambda \kappa}{\mu_0}}$, 
where $\kappa$ is the same as in the previous lemma.
\begin{lemma}
\label{lemma:wcc:upperboundepsmu}
	Under Assumptions~\ref{ass:cauchy_decrease}, \ref{ass:f}, and \ref{ass:J}, let 
	$j$ be a given iteration index such that for every $l \le j$ it holds that   
$\|F(x_l)\| \ge \max\left\{\epsilon,(1+\epsilon)\bar f\right\}$ 
	where $\epsilon \in ]0,1[$. 
	Then, for every $l \le j$, one also has
	$$	\mu_l \le \mu_{\max} :=
		\frac{\lambda\kappa}{\max\left\{\epsilon^2,(1+\epsilon)^2\bar f^2\right\}}. $$
\end{lemma}

\begin{proof}
	We prove this result by contradiction. Suppose that $l \ge 1$ is the first index such that
	\begin{equation}
	\label{eq:wcc:contradmu}
		\mu_{l} > \frac{\lambda\kappa}{\max\left\{\epsilon^2,(1+\epsilon)^2\bar f^2\right\}}.
	\end{equation}
	By the updating rules on $\mu_l$, 
	either the iteration $l-1$ is successful, in which case $\mu_l \le \mu_{l-1} \le \frac{\lambda\kappa}{\max\left\{\epsilon^2,(1+\epsilon)^2\bar f^2\right\}}$ which contradicts~\eqref{eq:wcc:contradmu}, or 
	the iteration $l-1$ is unsuccessful and thus  
	\[
		 \mu_{l} = \lambda \mu_{l-1} \quad \Rightarrow \quad \mu_{l-1} = \frac{\mu_l}{\lambda} > \frac{\kappa}{\max\left\{\epsilon^2,(1+\epsilon)^2\bar f^2\right\}} 
		> \frac{\kappa}{\|F(x_l)\|^2},
	\]
	therefore using Lemma \ref{lemma:wcc:condmusuccess} this implies that the $(l-1)$-th iteration is successful which leads to a contradiction again. 

\end{proof}
Thanks to Lemma~\ref{lemma:wcc:upperboundepsmu}, we can now bound the number of successful 
iterations needed to drive the gradient norm below a given threshold. 
\begin{proposition}
\label{propo:wcc:wccitssucc}
	Under Assumptions~\ref{ass:cauchy_decrease}, \ref{ass:f}, and \ref{ass:J}. Let 
	$\epsilon \in ]0,1[$ and $j_{\epsilon}$ be the first iteration index such that \eqref{eq:wcc:approxopt} holds.
	Then, if $\mathcal{S}_{\epsilon}$ is the set of indexes of successful iterations prior to $j_{\epsilon}$, 
	one has:
	\begin{equation*}
	%\label{eq:wcc:wccitssucc}
		\left| \mathcal{S}_{\epsilon} \right| \; \le \; \mathcal{C}\epsilon^{-2} ~~\mbox{with}~~\mathcal{C} := \frac{2 \left(\kappa_J^2+ \mu_{\max} \|F(x_0)\|^2\right)}
		{\eta \theta_{fcd}}f(x_0).
	\end{equation*}
\end{proposition}

\begin{proof}
	For any $j \in \mathcal{S}_{\epsilon}$, one has
	\begin{eqnarray*}
		f(x_j) - f(x_{j+1}) &\ge &\eta\left(m_j(0) - m_j(s_{j})\right)\ge 
		\eta \frac{\theta_{fcd}}{2}\frac{\|\nabla f(x_j)\|^2}
		{\kappa_J^2+\mu_j\|F(x_j)\|^2}.
	\end{eqnarray*}
	 Hence, using the fact that
	$\|\nabla f(x_j)\| \ge \epsilon$, $\|F(x_j)\| \le \|F(x_0)\|$ and  $\mu_j \le \mu_{\max}$, we arrive at
	\begin{eqnarray*}
		f(x_j) - f(x_{j+1}) 
		%&\ge &\eta \frac{\theta_{fcd}}{2}\frac{\|\nabla f(x_j)\|^2}
		%{\left(\kappa_J^2+\lambda\kappa\right)\tfrac{\|F(x_j)\|^2}{\epsilon^2}} \\
		&\ge &\eta \frac{\theta_{fcd}}{2}\frac{\epsilon^2}{\kappa_J^2+  \mu_{\max} \|F(x_0)\|^2}.
	\end{eqnarray*}
	Consequently, by summing on all iteration indices within $\mathcal{S}_{\epsilon}$,
	% and using the fact that $f$ is bounded below by $0$, 
	we obtain  
	\begin{equation*}
		f(x_0) - 0 \ge %\sum_{j=0}^{j_{\epsilon}} f(x_j) - f(x_{j+1}) \ge 
		\sum_{j \in \mathcal{S}_{\epsilon}} f(x_j) - f(x_{j+1}) \ge |\mathcal{S}_{\epsilon}|
		\frac{\eta\theta_{fcd}}{2\left({\kappa_J^2+ \mu_{\max} \|F(x_0)\|^2}\right)}
		\epsilon^2,
	\end{equation*}
	hence the result.

\end{proof}
\begin{lemma}
\label{lemma:wcc:unsucc}
	Under Assumptions~\ref{ass:cauchy_decrease}, \ref{ass:f}, and \ref{ass:J}. Let $\mathcal{U}_{\epsilon}$ denote the 
	set of unsuccessful iterations of index less than or equal to $j_{\epsilon}$. Then,
	\begin{equation}
	\label{eq:wcc:unsucc}
		| \mathcal{U}_{\epsilon} | \; \le \; 
		\log_{\lambda}\left( \frac{\kappa}{\mu_{\min}\epsilon^2}\right)\, |\mathcal{S}_{\epsilon}|.
	\end{equation}
\end{lemma}
\begin{proof}
	Note that we necessarily have $j_{\epsilon} \in \mathcal{S}_{\epsilon}$ (otherwise 
	it would contradict the definition of 
	$j_{\epsilon}$). 
	Our objective is to bound the number of unsuccessful iterations between two successful 
	ones. Let thus $\{j_0,\dots,j_t=j_{\epsilon}\}$ be an ordering of $\mathcal{S}_{\epsilon}$, and 
	$i \in \{0,\dots,t-1\}$.
	
	Due to the updating formulas for $\mu_j$ on successful iterations, we have:
	\[
		\mu_{j_i+1} \ge \max\{\mu_{\min},\bar{\mu}/\lambda\} \ge \mu_{\min}.
	\]	 
	Moreover, we have $\|F(x_{j_i+1})\| \ge \epsilon$ by assumption.
	By Lemma~\ref{lemma:wcc:condmusuccess}, for any unsuccessful iteration 
	$j \in \{j_i + 1,\dots, j_{i+1}-1\}$, we must have: $
		\mu_j \le \frac{\kappa}{\epsilon^2}$,
	since otherwise $\mu_j >  \frac{\kappa}{\epsilon^2} \ge \frac{\kappa}{\|F(x_j)\|^2}$ 
	and the iteration would be successful.
	
	Using the updating rules for $\mu_j$ on unsuccessful iterations, we obtain:
	\[
		\forall j=j_i + 1,\dots, j_{i+1}-1, \qquad 
		\mu_j = \lambda^{j-j_i-1}\mu_{j_i+1} \ge \lambda^{j-j_i-1}\mu_{\min}.
	\]
	Therefore, the number of unsuccessful iterations between $j_i$ and $j_{i+1}$, equal to 
	$j_{i+1}-j_i-1$, satisfies:
	\begin{equation}
	\label{eq:wcc:numberunsuccjiji1}
		j_{i+1}-j_i - 1 \; \le \; \log_{\lambda}\left( \frac{\kappa}{\mu_{\min}\epsilon^2}\right).
	\end{equation}
	By considering~\eqref{eq:wcc:numberunsuccjiji1} for $i=0,\dots,t-1$, we arrive at
	\begin{equation}
	\label{eq:wcc:sumunsuccbetweensucc}
		\sum_{i=0}^{t-1} (j_{i+1}-j_i - 1) 
		\; \le \; \log_{\lambda}\left( \frac{\kappa}{\mu_{\min}\epsilon^2}\right)
		\left[ \left|\mathcal{S}_{\epsilon}\right|-1\right].
	\end{equation}
	What is left to bound is the number of possible unsuccessful iterations between the 
	iteration of index $0$ and the first successful iteration $j_0$. Since $\mu_0 \ge \mu_{\min}$,
	a similar reasoning as the one used to obtain~\eqref{eq:wcc:numberunsuccjiji1} leads to
	\begin{equation}
	\label{eq:wcc:numberunsuccfrom0}
		j_0-1 \le \log_{\lambda}\left( \frac{\kappa}{\mu_{\min}\epsilon^2}\right).
	\end{equation}
	Putting~\eqref{eq:wcc:sumunsuccbetweensucc} and~\eqref{eq:wcc:numberunsuccfrom0} together yields the expected result.

\end{proof}
By combining the results from Proposition~\ref{propo:wcc:wccitssucc} and 
Lemma~\ref{lemma:wcc:unsucc}, we thus get the following complexity estimate.
\begin{theorem}
\label{theo:wcc:wccits}
	Under Assumptions~\ref{ass:cauchy_decrease}, \ref{ass:f}, and \ref{ass:J}. Let 
	$\epsilon \in ]0,1[$. Then, the first index $j_{\epsilon}$ for which 
%	$\|\nabla f(x_{j_{\epsilon}+1})\| < \epsilon$ or 
$
\|\nabla f(x_{j_{\epsilon}+1})\| < \epsilon ~~\mbox{or}~~~\|F(x_{j_{\epsilon}+1})\| < 
	\max\left\{\epsilon,(1+\epsilon)\bar f\right\},
$
	is bounded above by
	\begin{equation}
	\label{eq:wcc:wccits}
		\mathcal{C}\left( 1+ \log_{\lambda}\left[\frac{\kappa}{\mu_{\min}\epsilon^2}\right]\right)
		\epsilon^{-2},
	\end{equation}
	where $\mathcal{C}$ is the constant defined in Proposition~\ref{propo:wcc:wccitssucc}.
\end{theorem}
For the LM method proposed in this paper, we thus 
obtain an iteration complexity bound of $\tilde{\mathcal{O}}\left(\epsilon^{-2}\right)$, where 
the notation $\tilde{\mathcal{O}}(\cdot)$ indicates the presence of logarithmic factors in 
$\epsilon$. Note that the evaluation complexity bounds are of the same order.

We note that by the definition of $\bar f$ it holds that: $\|F(x_j)\|=\bar f$ implies that
$\nabla f(x_j)=0$. Thus, by letting $\epsilon \to 0$,   
Theorem \ref{theo:wcc:wccits} implies that
$$
\liminf_{j\to\infty} \|\nabla f(x_j)\|\; = \; 0.
$$
In order to derive the global convergence result, we need to extend this to a limit result.
\begin{theorem}
\label{th:gcv:lim}
 Under Assumptions \ref{ass:cauchy_decrease}, \ref{ass:f}, and ~\ref{ass:J}, the sequence $\{x_j\}$ generated by Algorithm \ref{alg:LM} satisfies
\[
\lim_{j\to\infty} \|\nabla f(x_j)\| \; = \; 0.
\]
\end{theorem}
\begin{proof}
Consider the case that $\mathcal{S}$, the set of successful iterations, is finite. Then $\exists j_0$ such that for all $j \ge j_0, ~ \|\nabla f(x_j)\|  =  \|\nabla f(x_{j_0})\|$ therefore from the previous theorem we conclude that in this case
\[
 \|\nabla f(x_{j_0})\| \; = \; \lim_{j\to\infty} \|\nabla f(x_j)\| \; = \;  \liminf_{j\to\infty} \|\nabla f(x_j)\| \; =  \; 0.
\]
Alternatively, assume $\mathcal{S}$ is infinite and let $\epsilon>0$.
Since $(f(x_j))$ is monotonically decreasing and bounded from below, one has 
 $\lim_{j\to\infty} f(x_j) - f(x_{j+1})=0$.
Since $\frac{\eta\theta_{fcd} \epsilon^2 }{2(\kappa_J^2 + \mu_0 \|F(x_0)\|^2 )}>0$, we conclude, for $j$ sufficiently large, that
 $$0 \le f(x_j) - f(x_{j+1}) \le \frac{\eta\theta_{fcd} \epsilon^2 }{2(\kappa_J^2 + \mu_0 \|F(x_0)\|^2 )}.$$
Thus, for a sufficiently large $j$, consider the case that $j \in  \mathcal{S}$, then $\rho_j \ge \eta$ and by rearranging the terms and using Assumption~\ref{ass:cauchy_decrease} we conclude that
\begin{eqnarray*}
 \|\nabla f(x_j)\| & \le & \sqrt{\frac{2(m_j(0)-m_j(s_j))(\kappa_J^2+\mu_0 \|F(x_0)\|^2)}{\theta_{fcd}}} \\
 & \le &  \sqrt{\frac{2(\kappa_J^2 + \mu_0 \|F(x_0)\|^2) \left(f(x_j) - f(x_{j+1})\right)}{\eta\theta_{fcd}}} \le \epsilon.
\end{eqnarray*}
If $j \notin  \mathcal{S}$ than
$
 \|\nabla f(x_j)\| \; = \; \|\nabla f(x_{\hat{j}})\| \;  \le \epsilon,
$
where $\hat{j} \in \mathcal{S}$ is the last successful iteration before $j$.
Hence, it must hold that $
\lim_{j\to\infty} \|\nabla f(x_j)\| \; = \; 0.$

\end{proof}
%%%%%%%%%%%%%%%%%%%%%%%%%%%%%%%%%%%%%%%%%%
\section{Local Convergence} \label{sec:local} 
%%%%%%%%%%%%%%%%%%%%%%%%%%%%%%%%%%%%%%%%%%%
In this section, we prove local convergence for the Algorithm, showing a quadratic rate for zero residual problems
and explicit linear rate for nonzero residuals.
Since the problem is generally nonconvex and with possibly nonzero residuals, there can be multiple 
sets of stationary points with varying objective values. We consider a particular subset
with a constant value of the objective.
\begin{assumption}\label{as:soln}
There exists a connected isolated set $X^*$ composed of stationary points to~\eqref{function-f} each  
with the same value $f^*$ of $f(\cdot)=\frac{1}{2}\|F(\cdot)\|^2$ and 
Algorithm~\ref{alg:LM} generates a sequence with an accumulation point $x^*\in X^*$. 
\end{assumption}
We note, that under Assumption \ref{as:soln}, the value 
$\|F({\bar x})\|$ is unique for all $\bar x\in X^*$ which may not be the case for the residual vector $F({\bar x})$.
Thus we define $\bar f = \sqrt{2f(\bar x)}=\|F(\bar x)\|$ for $\bar x\in X^*$.
Henceforth, from the global convergence analysis, we can assume, without loss of generality, that there exists
a subsequence of iterates approaching this $X^*$. This subsequence does not need to be unique, i.e., there may be more than one  
subsequence converging to separate connected sets of stationary points. We shall see that eventually,
one of these sets shall ``catch'' the subsequence and result in direct convergence to the set of stationary points
at a quadratic or linear rate, depending on $\bar f$.

In the sequel, $N(x,\delta)$ denotes  the closed ball with center $x$ (a given vector) and radius $\delta>0$ and
$\dist(x,X^*)$ denotes the distance between the vector $x$ and the set $X^*$, i.e., $\dist(x,X^*)=\min_{y\in X^*}\|x-y\|,$ and 
${\bar x} \in \argmin_{y\in X^*}\|x-y\|$.
\begin{assumption}\label{as:lip}
It holds that $F(x)$ is twice continuously differentiable around $x^*\in X^*$ with
$x^*$ satisfying Assumption~\ref{as:soln}. 
 In particular
this implies, 
  that there exists
$\delta_1>0$ such that for all $x, ~y \in N(x^*,\delta_1)$,
\begin{equation}\label{lipgrad}
\|\nabla f(x)\|^2  =  \| J(x)^{\top}F(x)-J({\bar x})^{\top}{F(\bar x)}\|^2 \le L_1 \|x-\bar x\|^2,  
\end{equation}
\begin{equation}\label{lipf1}
\|F(x)-F(y)\| \le L_2 \|x-y\|,
\end{equation}
\begin{equation}\label{lipf_J}
\| F(y) -F(x) - J(x)(y-x)\| \le L_3 \| x-y\|^2,
\end{equation}
where $L_1$, $L_2$, and $L_3$ are positive constants.
\end{assumption}
From the triangle inequality and assuming (\ref{lipf1}), we get
\begin{equation}\label{lipf2}
\|F(x)\|-\|F(y)\|  \le    \|F(x)-F(y)\|   \le L_2 \|x-y\|.
\end{equation}
We introduce the following additional assumption.
\begin{assumption}\label{as:errf}
There exist a $\delta_3>0$ and $M >0$ such that
\[
\forall x\in N(x^*,\delta_3),~~~~\dist(x,X^*)\le M \|F(x)- F(\bar x)\|.
\]
\end{assumption}
As the function $x \rightarrow F(x)- F(\bar x)$ is zero residual, the proposed error bound assumption can be seen as a generalization of the zero residual case~\cite{yamashita2001rate,fan2005quadratic,dan2002convergence,facchinei2013family}. Thus any ill-posed zero residual problem, as considered
in this line of work on quadratic local convergence for LM methods, satisfies the assumptions. Our assumptions are also covered by a range of nonzero residual problems, for 
instance, standard data assimilation problems~\cite{YTremolet_2007} as given by Example \ref{example:1}.    
\begin{example}
\label{example:1} 
Consider the following data assimilation problem $F: \mathbb{R}^n \to \mathbb{R}^{m}$ defined,  for a given $x \in \mathbb{R}^n$, by 
$F(x)= \left( (x - x_{\mbox{b}})^{\top}, (H(x) - y)^{\top}\right)^{\top}$,
where $x_{\mbox{b}} \in \mathbb{R}^n$ is a background vector, $y \in \mathbb{R}^{m-n}$ the vector of observations and $H: \mathbb{R}^n \to \mathbb{R}^{m-n}$ is a smooth operator modeling the observations. For such problems, the set of stationary points $X^{*}=\{\{ \bar x\}\} \subset \mathbb{R}^n$ is a finite disjoint set, $\bar F=F(\bar x)$ and $ \dist(x,X^*)= \|x - \bar x\|$ for $\bar{x}$ closest to $x$.
   Clearly, one has 
\begin{eqnarray*}
\|F(x) - F(\bar x)\|^2& = &\|x - \bar x\|^2 + \|H(x) - H(\bar x)\|^2 \ge  \|x - \bar x\|^2  = \; \dist(x,X^*)^2.
\end{eqnarray*} 
Thus, for these typical problems arising in data assimilation, Assumptions \ref{as:lip}, \ref{as:soln} and \ref{as:errf}  are satisfied.\end{example}
\begin{example}
\label{example:2}
Consider $F: \mathbb{R}^3 \to \mathbb{R}^3$ defined,  for a given $x \in \mathbb{R}^3$, by $$F(x)= (\exp(x_1 - x_2) -1, x_3-1, x_3 +1)^{\top}.$$  Clearly, $F$ is 
Lipschitz smooth, hence Assumption \ref{as:lip} is satisfied.   
We note that for all $x$, $\bar{x}\in X^*=\{x \in \mathbb{R}^3: x_1= x_2 ~\mbox{and}~ x_3=0 \}$ satisfies $F(\bar x)=(0, -1, 1)^{\top} $  and, for a given $x$, one has  
$\dist(x,X^*)= \sqrt{\frac{(x_1-x_2)^2}{2} + x_3^2}.$ Hence,  for all $x \in \mathbb{R}^3$ such that $|x_1-x_2 |\le \frac{1}{2}$, one concludes that
\begin{eqnarray*}
\|F(x) - F(\bar x)\|& = &\sqrt{(\exp(x_1 - x_2) - 1)^2 +  2 x^2_3}  \ge   \; \dist(x,X^*).
\end{eqnarray*} 
Thus, for this example, Assumptions \ref{as:soln} and \ref{as:errf}  are also satisfied.\end{example}
From the global convergence results, we have established that there is a
subsequence of successful iterations converging to the set of stationary points $X^*$. In this section, we begin
by considering the subsequence of iterations that succeed the successful iterations, i.e., we consider
the subsequence $\mathcal{K}=\{j+1 : \, j\in\mathcal{S}\}$. 
We shall present the results with a slight abuse of notation that simplifies the presentation without sacrificing accuracy or generality:
in particular every time we denote a quantity $a_j$, the index $j$ corresponds
to an element of this subsequence $\mathcal{K}$ denoted above, thus when we say a particular statement holds eventually,
this means that it holds for all $j\in \mathcal{S}+1$ with $j$ sufficiently large. Let $\hat \mu$ be an upper bound for $ \mu_j$, note that this exists by the formulation of Algorithm~\ref{alg:LM}. 
We shall denote also $\delta$ as $\delta = \min(\delta_1,\delta_2,\delta_3)$,
with $\{\delta_i\}_{i=1,2,3}$. In the proof, we follow the structure of the local convergence proof in~\cite{yamashita2001rate}, with the addition
that the step is accepted by the globalization procedure. 
We use ${\bar F_j}$ to denote $F({\bar x_j})$. 
The first lemma is similar to~\cite[Lemma 2.1]{yamashita2001rate}. 
\begin{lemma}\label{lem:distsubp}
Suppose that Assumptions \ref{as:soln} and \ref{as:lip} are satisfied.\\
 If $x_j\in N(x^*,\frac{\delta}{2})$, then the solution $s_j$ to~\eqref{eq:LMsubproblem} satisfies
$$% \begin{equation}\label{eq:fddist}
\|J_js_j + F_j\| - \bar f \le  C_1\dist(x_j,X^*)^2,
$$%\end{equation}
where $C_1$ is a positive constant independent of $j$.
\end{lemma}
\begin{proof}
Let us assume that $\bar f>0$. Otherwise the proof is the same as in~\cite[Lemma 2.1]{yamashita2001rate}.
Without loss of generality, since $f(x_j)$ is monotonically decreasing, we can consider that $j$ is sufficiently large
such that $\|F_j\|\le 2 \bar f$.
Hence, we get 
\begin{eqnarray*}
  \|J_js_j+F_j\|^2 &\le&  2 m_j(\bar x_j-x_j) 
  =  \|J_j(\bar x_j-x_j)+F_j\|^2 +\mu_j\|F_j\|^2\|x_j-\bar x_j\|^2 \\ 
    &\le& \left( \|{\bar F_j}\|  + \|J_j(\bar x_j-x_j)+F_j - {\bar F}_j \|\right)^2 +\mu_j\|F_j\|^2\|x_j-\bar x_j\|^2 \\ 
 & \le&   \left(\|{\bar F_j}\|+L_3\|x_j-\bar x_j\|^2\right)^2+4\mu_j\|\bar F_j\|^2\|x_j-\bar x_j\|^2 \\ 
  & =& L_3^2\|x_j-\bar x_j\|^4+2\|{\bar F}_j\|(2 \mu_j \|{\bar F}_j\|+L_3 )\|x_j-\bar x_j\|^2+\|{\bar F_j}\|^2 \\ 
 & \le &\left(\left(2 \hat \mu \bar f+L_3 \right)\|x_j-\bar x_j\|^2+\bar f\right)^2,
\end{eqnarray*}
which concludes the proof.

\end{proof}
\begin{lemma}\label{lem:distsoln} 
Suppose that Assumptions \ref{as:soln} and \ref{as:lip} are satisfied.\\
If $x_j\in N(x^*,\frac{\delta}{2})$, then the solution $s_j$ to~\eqref{eq:LMsubproblem} satisfies
\begin{equation}\label{eq:sdist}
\|s_j\| \le C_2 \dist(x_j,X^*),
\end{equation}
where $C_2$ is a positive constant independent of $j$.
\end{lemma}
\begin{proof}
We assume that $\|\bar F_j\|=\bar f>0$, otherwise the proof is the same as in~\cite[Lemma 2.1]{yamashita2001rate}.
In this case, using the fact that $\|F_j\|\ge \bar f$ and $\mu_j \ge \mu_{\min}$, one has
$$
\|s_j\| \le \frac{\|\nabla f(x_j)\|}{\gamma_j}=\frac{\|\nabla f(x_j)\|}{\mu_j \|F_j\|^2} \le \frac{\sqrt{L_1}}{\mu_{\min} \bar f^2} \dist(x_j,X^*),
$$
which concludes the proof.

\end{proof}
\begin{lemma}\label{thm:stepaccept}
Suppose that Assumptions \ref{as:soln}, \ref{as:lip} and \ref{as:errf} are satisfied. Consider the case where $\bar f=0$, then for $j$ sufficiently large, one has $\rho_j\ge \eta$.
\end{lemma}
\begin{proof}
In fact, for $j$ sufficiently large, using Lemmas~\ref{lem:distsubp} and~\ref{lem:distsoln}, one gets

$
\begin{array}{l}
2 \left(m_j(0) - m_j(s_j)\right)  =   \|F_j\|^2-\|F_j+J_j s_j\|^2-\gamma_j\|s_j\|^2 \\
\ge  \left(\|F_j\|+\|F_j+J_js_j\|\right)\left(\frac{1}{M}\|x_j-\bar x_j\|-C_1 \|x_j-\bar x_j\|^2 \right)-C^2_2\gamma_j \|x_j-\bar x_j\|^2 \\
\ge \|F_j\|\left(\frac{1}{M}\|x_j-\bar x_j\|-C_1 \|x_j-\bar x_j\|^2 \right)-C^2_2\mu_j \|F_j\|^2\|x_j-\bar x_j\|^2 \\
\ge \frac{1}{M}\|x_j-\bar x_j\| \left(\frac{1}{M}\|x_j-\bar x_j\|-C_1 \|x_j-\bar x_j\|^2 \right)-L^2_2 C^2_2 {\hat \mu} \|x_j-\bar x_j\|^4 \\
= \frac{1}{M^2}\|x_j-\bar x_j\|^2- \frac{C_1}{M}\|x_j-\bar x_j\|^3-L^2_2 C^2_2 {\hat \mu} \|x_j-\bar x_j\|^4>0. 
\end{array}
$

On the other hand, using the fact that $\|J_j\|$ is bounded (by  $\kappa_J>0$), Lemma~\ref{lem:distsoln} and (since $\bar{F}_j=0$)
$\frac{1}{M} \|\bar x_j-x_j\| \le  \|F_j\| \le L_2\|\bar x_j-x_j\|$, one gets 

$
\begin{array}{l}
2\left|f(x_j+{s_j}) - m_j(s_j)\right|  =  \left|\|F(x_j+s_j)\|^2-\|F_j+J_j s_j\|^2-\gamma_j\|s_j\|^2\right|  \\
 = \left|\left(\|F(x_j+s_j)\|-\|F_j+J_j s_j\|\right)\left(\|F(x_j+s_j)\|+\|F_j+J_j s_j\|\right)-\gamma_j\|s_j\|^2\right|  \\
  \le  L_3 \|s_j\|^2\left(\|F(x_j+s_j)\|+\|F_j\|+\|J_j\|\|s_j\|\right)+ \gamma_j\|s_j\|^2%C_2 {\mu_j} \|F_j\|^2| \|x_j-\bar x_j\|^2
   \\
    \le  L_3 \|s_j\|^2\left( L_2 \|x_j+s_j-\bar x_j\|+ L_2 \|x_j-\bar x_j\|+ \|J_j\|\|s_j\|\right)+\gamma_j\|s_j\|^2\\
    \le (C_2 L_2+ 2L_2+C_2 \kappa_{J}) C_2^2L_3 \|x_j-\bar x_j\|^3+L_2^2 C_2^2\hat \mu \|x_j-\bar x_j\|^4.
\end{array}
$

Hence, for $j$ sufficiently large
\begin{eqnarray*}  
|1-\rho_j| &= & \left|\frac{m_j(0)-f(x_j) + f(x_j+{s_j}) - m_j(s_j)}{m_j(0)-m_j(s_j)} \right|\\ 
& \le &\frac{(C_2 L_2+ 2L_2+C_2 \kappa_{J}) C_2^2L_3 \|x_j-\bar x_j\|+L_2^2 C_2^2\hat \mu \|x_j-\bar x_j\|^2}{\frac{1}{M^2}- \frac{C_1}{M}\|x_j-\bar x_j\|-L^2_2 C^2_2 {\hat \mu} \|x_j-\bar x_j\|^2}. %\\
%& \to &0 ~~~\mbox{when $j$ goes to $+\infty$}.
\end{eqnarray*}
Thus, $|1-\rho_j| \to 0$ as $j$ goes to $+\infty$.

\end{proof}
For the nonzero residual case, we must consider a specific instance of the Algorithm. In particular, we specify Step 3 of Algorithm~\ref{alg:LM}
to be,
\begin{center}
If $\rho_j\ge \eta$, then set $x_{j+1}=x_j+s_j$,  $\mu_{j+1}=\bar{\mu}$ and $\bar{\mu}=\mu_{j}$.  
\end{center}
Note that this is still consistent with the presentation of the Algorithm.
\begin{lemma}\label{thm:stepaccept:nnz}
Suppose that Assumptions \ref{as:soln}, \ref{as:lip} and \ref{as:errf} are satisfied. Consider that $\bar f>0$, then for $j$ sufficiently large, one has $\rho_j\ge \eta$.
\end{lemma}
\begin{proof}
Indeed, by the new updating mechanism, the parameter $\bar{\mu}$ is monotonically nondecreasing. In particular,
if there is an infinite set of unsuccessful steps, then ${\mu_j}\to \infty$. %, and by the properties of the algorithm, $\mu_j\ge \bar{\mu}\to \infty$. But 
This implies that for some $j_0$ it holds that for $j\ge j_0$, $\mu_j > \frac{\kappa}{\bar{f}^2} > \frac{\kappa}{\|F(x_j)\|^2}$, which together
with Lemma~\ref{lemma:wcc:condmusuccess} reach a contradiction. Thus, there is a finite number of unsuccessful steps, and 
every step is accepted for $j$ sufficiently large. 

\end{proof}
\begin{proposition}\label{lem:dist}
Suppose that Assumptions~\ref{as:soln},~\ref{as:lip}, and~\ref{as:errf} are satisfied. Let $x_j$, $x_{j+1} \in N(x^*,\delta/2)$.
One has 
\begin{equation*}
\left (1-   \sqrt{m} L_1 M^{2} \bar f\right)\dist(x_{j+1},X^*)^2 \le  C_3^2 \dist(x_{j},X^*)^4 + {\hat C_3}^2 \bar f \dist(x_{j},X^*)^2,
\end{equation*}
where $C_3:= M \sqrt{C_1^2 + 2L_3C_1 C_2^2 + L_3^2C_2^4}$ and $\hat C_3:=M  \sqrt{ 2C_1 + 2L_3C_2^2}$.
\end{proposition}
\begin{proof}
Indeed, using Assumption~\ref{as:errf}, Lemma \ref{lem:distsubp}, the fact that
the step is accepted for $j$ sufficiently large, and $\bar f =\|{\bar F_{j+1}}\|$, one has

$
\begin{array}{l}
\|x_{j+1}-\bar x_{j+1}\|^2 \le M^{2}\|F(x_j+s_j)-{\bar F}_{j+1}\|^2 \\
 \le M^{2} \left( \|F(x_j+s_j)\|^2-2F(x_j+s_j)^{\top} {\bar F}_{j+1}+\bar f^2\right)\\
  \le  M^{2} \left( \left(\|J(x_{j})s_{j}+F_j\|+L_3 \|s_{j}\|^2\right)^2 -2F(x_j+s_j)^{\top} {\bar F}_{j+1}+\bar f^2\right)\\
  \le  M^{2} \left( \left \|J(x_{j})s_{j}+F_j \right\|^2+2L_3\|J(x_{j})s_{j}+F_j\|\|s_{j}\|^2 +L^2_3 \|s_{j}\|^4 \right. \\
  \left. \quad -2F\left(x_j+s_j\right)^{\top} {\bar F}_{j+1}+\bar f^2 \right)\\
   \le M^{2} \left(  C_1^2\left \|x_j-\bar x_j\right \|^4+2C_1\|x_j-\bar x_j\|^2\bar f+\bar f ^2+2L_3C_1\|x_j-\bar x_j\|^2\|s_{j}\|^2\right.\\ 
   \left. \quad+ 2L_3\bar f \|s_{j}\|^2 +L^2_3 \|s_{j}\|^4 -2F\left(x_j+s_j\right)^{\top} {\bar F}_{j+1}+\bar f^2 \right).
\end{array}
$

Therefore, using Lemma  \ref{lem:distsoln}, one gets
{\small{
\begin{eqnarray}
\label{eq:conv}   
\|x_{j+1}-\bar x_{j+1}\|^2 \le  C_3^2 \|x_j-\bar x_j\|^4 + {\hat C_3}^2 \bar f \|x_j-\bar x_j\|^2 + 2 M^{2} |F(x_j+s_j)^{\top} {\bar F}_{j+1}-\bar f^2|,
\end{eqnarray}
}}
where  $C_3 := M \sqrt{C_1^2 + 2L_3C_1 C_2^2 + L_3^2 C_2^4}$ and $ \hat C_3:=M  \sqrt{ 2C_1 + 2L_3C_2^2}$ are positive constants.
Moreover, by applying Taylor expansion to $x \rightarrow F(x)^{\top} {\bar F}_{j+1}$ at the point $x_{j+1}=x_j+s_j$ around $\bar x_{j+1}$, there exists $R>0$ such that
 \begin{eqnarray*}
 \label{eq:hess}
 |F_{j+1}^{\top} {\bar F}_{j+1} -  \bar f ^2 | &\le &|{(J(\bar x_{j+1})^{\top}\bar F_{j+1})}^{\top} (x_{j+1}-\bar x_{j+1})|+ R  \|x_{j+1}-\bar x_{j+1}\|^2 \\
 &=& |{\nabla f(\bar x_{j+1})}^{\top} (x_{j+1}-\bar x_{j+1})| + R  \|x_{j+1}-\bar x_{j+1}\|^2 \\
 &=& R  \|x_{j+1}-\bar x_{j+1}\|^2 ~~~~\mbox{by using $\bar x_{j+1} \in X^*$}.
\end{eqnarray*} 
 %{\color{red}{ 
Note that the Hessian of $x \rightarrow F(x)^{\top} {\bar F}_{j+1}$ is equal to $\sum_{i=1}^m \bar F(x_{j+1})_i \nabla^2 F_i(x)$, and
 from Assumption \ref{as:lip} we have $\nabla^2 F_i(x)$ are bounded. Hence, the constant $R$ is bounded as follows
$R \le \frac{L_1}{2} \sum_{i=1}^m |(\bar  F_{j+1})_i| \le  \frac{\sqrt{m} L_1}{2} \|{\bar F}_{j+1}\|. $
%Note that, by using Assumption \ref{as:lip}, the gradient of $x \rightarrow F(x)^{\top} {\bar F}$ is locally Lipschitz. Hence, the constant $R$ can be bounded as follows $$R \le L_1 \sum_{i=1}^m |\bar F_i| \le  \sqrt{m} L_1 \bar f . $$
%}}
Combining the obtained Taylor expansion and (\ref{eq:conv}) gives
$$
\|x_{j+1}-\bar x_{j+1}\|^2 \le  C_3^2 \|x_j-\bar x_j\|^4 + {\hat C_3}^2 \bar f  \|x_j-\bar x_j\|^2  +   \sqrt{m} L_1 M^{2} \bar f   \|x_{j+1}-\bar x_{j+1}\|^2,
$$
which completes this proof.

\end{proof}
In next lemma, we show that, once the iterates $\{x_j\}_j$ lie sufficiently near their limit point $x^*$, the sequence $\{\dist(x_{j},X^*)\}_j$ converges to zero quadratically if the problem has a zero residual, or linearly when the residual is small.
\begin{lemma}\label{lem:quaddist}
Suppose that Assumptions~\ref{as:soln},~\ref{as:lip}, and~\ref{as:errf} are satisfied. Let $\{x_j\}_j$ be a sequence generated by the proposed Algorithm. Suppose that both $x_j$ and $x_{j+1}$ belong to $N(x^*,\delta/2)$. If the problem has a zero residual, i.e., $\bar f=0$, then
\begin{equation} \label{eq:quaddist}
\dist(x_{j+1},X^*)\le C_3\dist(x_{j},X^*)^2,
\end{equation}
where $C_3$ is a constant defined according to Proposition \ref{lem:dist}.

%If the problem has a small nonzero residual, i.e., 
Otherwise, if
$\bar f < \min\left\{\frac{1}{\sqrt{m} L_1 M^{2}},\frac{1-C^2_3\delta^2}{{\hat C_3}^2 + \sqrt{m} L_1 M^{2}}\right\},$ then 
\begin{equation} \label{eq:lindist}
\dist(x_{j+1},X^*)\le C_4\dist(x_{j},X^*),
\end{equation}
where $C_4\in ]0,1[$ is a positive constant independent of $j$.
\end{lemma}
\begin{proof}
Indeed, under the zero residual case, i.e., $\bar f=0$, then Proposition \ref{lem:dist} is equivalent to
$
\dist(x_{j+1},X^*) \le  C_3 \dist(x_{j},X^*)^2.$

If the problem has a small residual, i.e., $\bar f < \min\left\{\frac{1}{\sqrt{m} L_1 M^{2}},\frac{1-C^2_3\delta^2}{{\hat C_3}^2 + \sqrt{m} L_1 M^{2}}\right\},$
then Proposition \ref{lem:dist} will be equivalent to
 $$
\dist(x_{j+1},X^*)^2 \le  \frac{C_3^2 \delta^2  + {\hat C_3}^2 \bar f }{1 - \sqrt{m} L_1 M^{2}  \bar f} \dist(x_{j},X^*)^2 = C^2_4\dist(x_{j},X^*)^2,$$
where $C_4:=\sqrt{ \frac{C_3^2 \delta^2  + {\hat C_3}^2 \bar f  }{1 - \sqrt{m} L_1 M^{2}  \bar f }} \in ]0,1[$.

\end{proof}
The final theorem is standard (see, e.g.,~\cite[Lemma 2.3]{yamashita2001rate}). %and thus we omit the proof.
\begin{theorem}\label{th:quadconv}
Suppose that  Assumptions~\ref{as:soln},~\ref{as:lip}, and~\ref{as:errf} are satisfied.
If $\bar f=0$ then Algorithm~\ref{alg:LM} converges locally quadratically to $X^*$.
Otherwise, if the problem has a small nonzero residual as in Lemma \ref{lem:quaddist}, 
Algorithm~\ref{alg:LM} converges locally at a linear rate to $X^*$.
\end{theorem}
\section{Numerical Results}
\label{sec:num}
In this section, we report the results of some preliminary experiments performed to test the practical behavior of Algorithm~\ref{alg:LM}. All procedures were implemented in
Matlab and run using Matlab 2019a on a MacBook Pro 2,4 GHz Intel Core i5, 4 GB RAM; the machine precision is $\epsilon_m \sim2\cdot 10^{-16}$.

We will compare our proposed algorithm with the LM method proposed in~\cite{RZhao_JFan_2016}. In fact, the latter algorithm can be seen 
to be similar to Algorithm~\ref{alg:LM} except that $\gamma_j := \mu_j \|\nabla f(x_j)\|$ where the parameter $\mu_j$ is updated in the following way. Given some constants $c_0>1$, $ c_1= c_0(c_0 +1)^{-1}$ and $0< \eta < \eta_1 < \eta_2<1$, if the iteration is
unsuccessful then $\mu_{j}$ is increased (i.e., $\mu_{j+1}:= c_0 \mu_{j}$). Otherwise, if $\|\nabla f(x_j)\| \mu_j < \eta_1$ then $\mu_{j+1}= c_0 \mu_j$, if $\|\nabla f(x_j)\| \mu_j>\eta_2$ then $\mu_{j+1}=\max\{c_1 \mu_j, \mu_{\min} \} $, and $\mu_j$ is kept unchanged otherwise.
The LM method proposed in~\cite{RZhao_JFan_2016} was shown to be globally convergent with a favorable complexity bound but its local behavior was not investigated.  In our comparison, we will refer to the implementation of this method as \texttt{LM-(global)}, while Algorithm~\ref{alg:LM} will be referred to as  \texttt{LM-(global and local)} (since it theoretically guarantees both global and local convergence properties). %, we will refer to it with \texttt{LM-(global and local)}.
Both algorithms were written in Matlab and the subproblem was solved using the backslash operator. For the \texttt{LM-(global and local)} method, two variants were tested. In the first one, named \texttt{LM-(global and local)-V1}, we set the parameter $\mu_{j+1}$ equal to $\max\{{\bar \mu}/\lambda, \mu_{\min}\}$ if the iteration is declared successful. In the second variant, named \texttt{LM-(global and local)-V2}, the parameter is set $\mu_{j+1}=\bar \mu$ if the iteration $j$ is successful.
The initial parameters defining the implemented algorithms were set as:
$
\eta=10^{-2}, ~ \eta_1=0.1,  ~ \eta_2=0.9, ~\lambda= c_0=5, ~\mu_{0}=1$ and $\mu_{\min}=10^{-16}.  
$
As a set of problems $\mathcal{P}$, we used the well known $33$ Mor\'e/Garbow/Hillstrom problems \cite{More81b}. All the tested problems are smooth and have a least-squares structure. The residual function $F$ and  the Jacobian matrix for all the test problems~\cite{More81b} are implemented in Matlab. 
Some of these problems have a nonzero value at the optimum and thus are consistent with the scope of the paper. To obtain a larger test set, we created a set of additional $14$ optimization problems by
varying the problem dimension $n$ when this was possible.  For all the tested problems, we used the proposed starting points $x_0$ given in the problems' original presentation~\cite{More81b}.
All algorithms are stopped when
$
\|\nabla f(x_j)\|\le \epsilon$ where $\epsilon$ is the regarded accuracy level.
If they did not converge within a maximum number of iterations $j_{\max}:=10000$, then they were considered to have failed. 

For our test comparison, we used the performance profiles proposed by Dolan and Mor\'e~\cite{Dolan_2002} over the set of problems $\mathcal{P}$ (of cardinality $|\mathcal{P}|$).
For a set of algorithms $\mathcal{S}$, the performance profile
$\rho_s(\tau)$ of an algorithm~$s$ is defined as the fraction of problems
where the performance ratio $r_{p,s}$ is at most $\tau$,
$
 \rho_s(\tau) \; = \; \frac{1}{|\mathcal{P}|} \mbox{size} \{ p \in \mathcal{P}: r_{p,s} \leq \tau \}.$
The performance ratio $r_{p,s}$ is in turn defined by
$r_{p,s} \; = \; \frac{t_{p,s} }{\min\{t_{p,s}: s \in \mathcal{S}\}},$ 
where $t_{p,s} > 0$ measures the performance of the algorithm~$s$ when solving problem~$p$, seen here as the number of iterations. 
Better performance of the algorithm~$s$,
relatively to the other algorithms on the set of problems,
is indicated by higher values of $\rho_s(\tau)$.
In particular, efficiency is measured by $\rho_s(1)$ (the fraction of problems for which algorithm~$s$ performs the best) and robustness is measured by $\rho_s(\tau)$ for $\tau$ sufficiently large
(the fraction of problems solved by~$s$). For a better visualization,
we plot the performance profiles in a $\log_2$-scale.

\begin{figure}[!ht]
\centering
\includegraphics[scale=0.45]{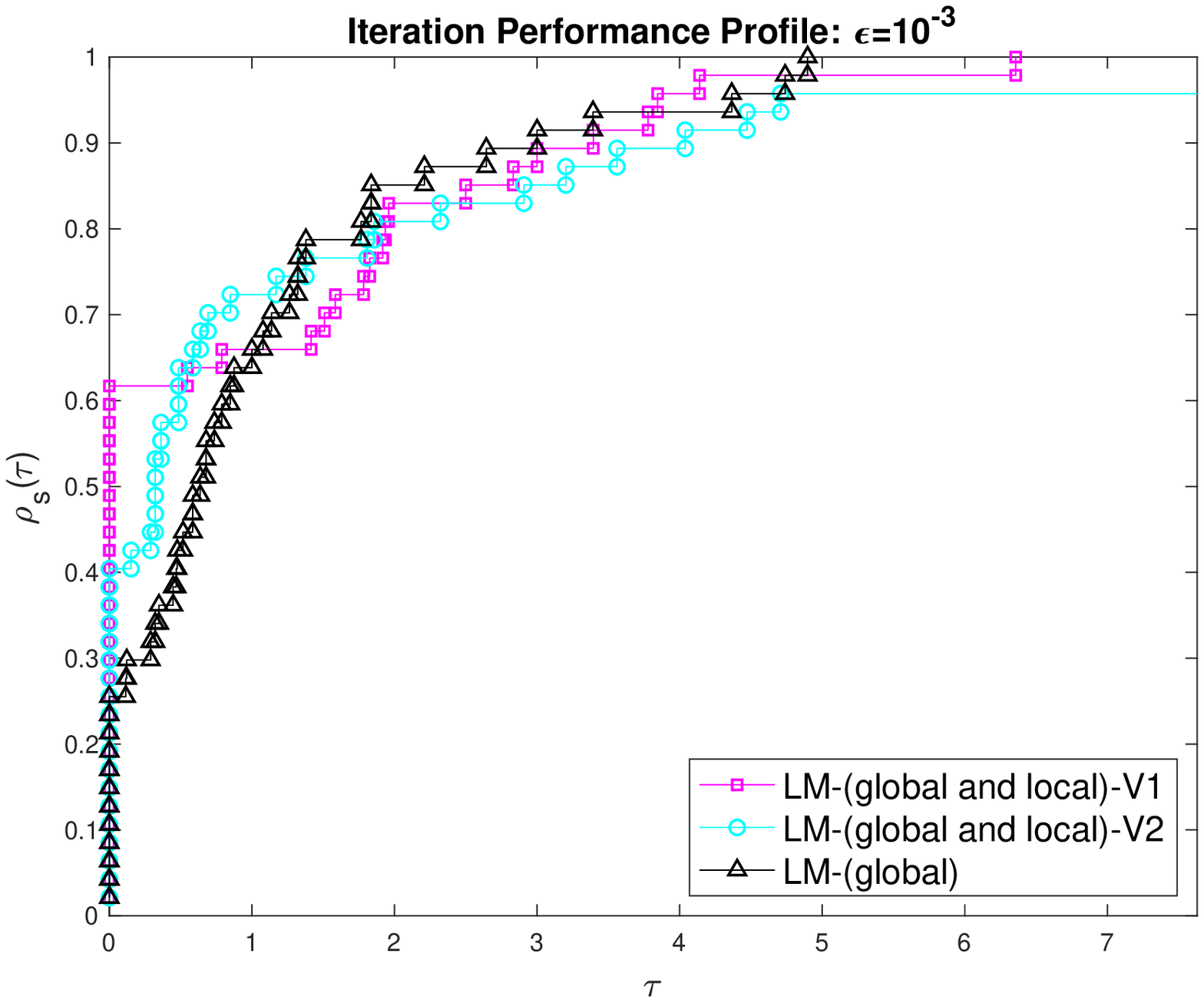} \label{subfig1:pp}
\includegraphics[scale=0.45]{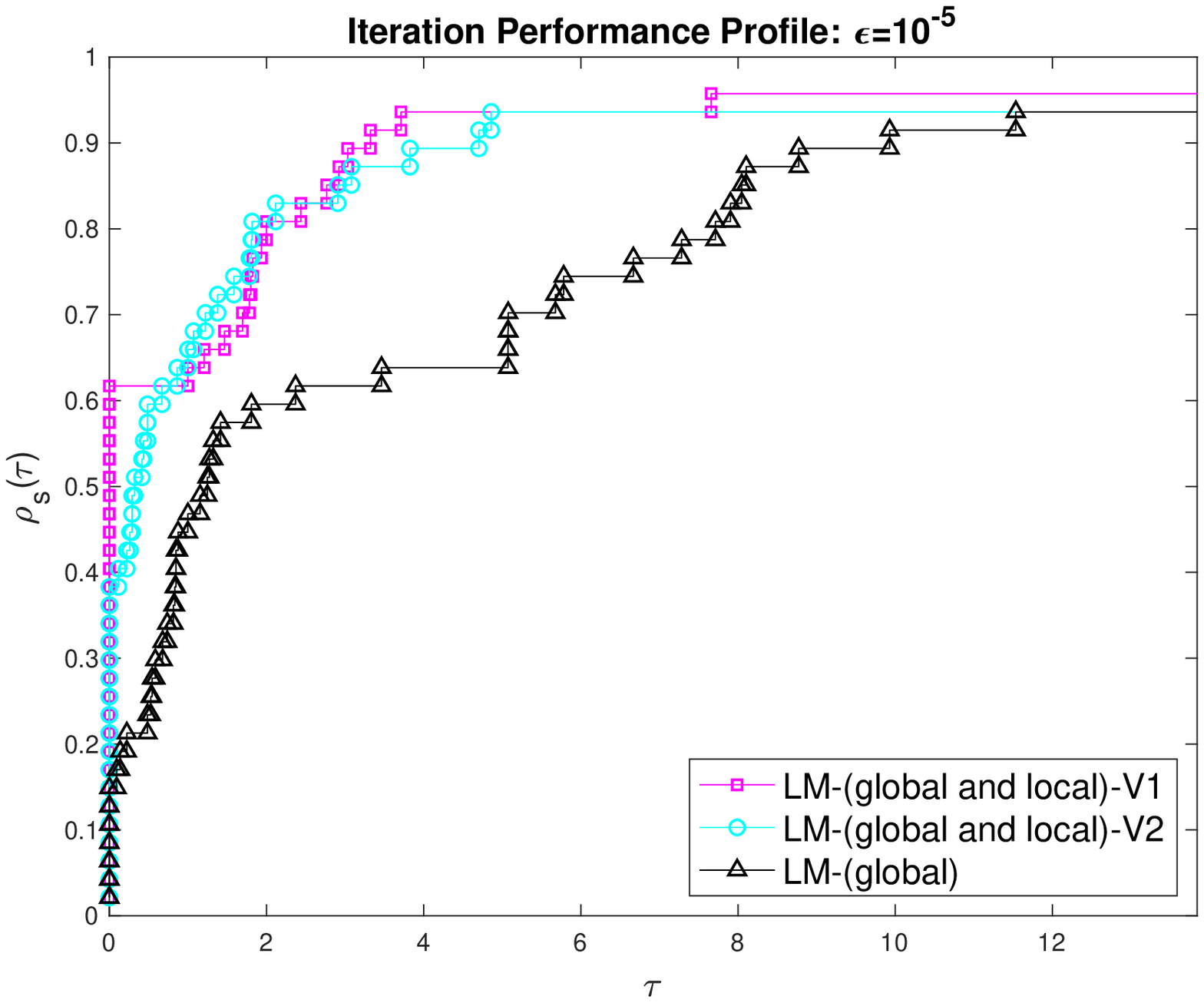}\label{subfig2:pp}
\caption{Obtained performance profiles considering the
two levels of accuracy, $10^{-3}$ and $10^{-5}$.} \label{fig:pp}
\end{figure}
We present the obtained performance profiles using two levels of accuracy in Figure \ref{fig:pp}. For a level of accuracy of $10^{-3}$,  \texttt{LM-(global and local)} variants present a better efficiency compared to \texttt{LM-(global)} (in more than $60 \%$ of the tested problems \texttt{LM-(global and local)-V1}  performed best, and  \texttt{LM-(global and local)-V2} performed better on  $40 \%$ while  \texttt{LM-(global)} was better on less than $30 \%$). When it comes to robustness, all the solvers exhibit good performance. Using a higher accuracy, the two variants of  \texttt{LM-(global and local)} outperform \texttt{LM-(global)}.  The \texttt{LM-(global and local)-V1} variant shows the best performance both in terms of efficiency and robustness.

In order to estimate the local convergence rate, we estimated the order of convergence  by
$$
  \mbox{\text{EOC}} := \log \left(\frac{\|\nabla f(x_{j_f})\|}{\max\{1,\|\nabla f(x_{0})\|\}} \right)/ \log \left(\frac{\|\nabla f(x_{j_f-1})\|}{\max\{1,\|\nabla f(x_{0})\|\}} \right),
$$
where $j_f$ is the index of the final computed iterate. When $ \mbox{\text{EOC}} \ge 1.8$,  the algorithm will be said quadratically convergent. If $ 1.8 >\mbox{\text{EOC}} \ge 1.1$, then  the algorithm will  be seen as superlinearly convergent. Otherwise, the algorithm is  linearly convergent or worse. 
\begin{table}[!h]
\centering
\caption{Order of convergence for problems in $\mathcal{P}$ with the accuracy level $\epsilon=10^{-5}$.}
\label{table:eoc:P1}
\begin{tabular}{|c|l|c|c|c|}
\hline%{2-4} 
% & \multicolumn{3}{l}{EOC using GLC-LM} & \multicolumn{3}{l}{EOC using GC-LM} \\
\multirow{2}{*}{$\mathcal{P}$ }  & \multirow{2}{*}{Method} & \multicolumn{3}{c|}{Number of problems to converge}  \\  \cline{3-5} 
 &                                       &  Linear or worse    &    Superlinear &  Quadratic    \\ \hline \hline
%Well-conditioned         &     2  & 7           &    16  &           &       &      \\
%Poorly conditioned      &    0   & 1           &     2   &           &       &    \\
zero  & \multicolumn{1}{|l|}{\texttt{LM-(global and local)-V1}  }           &   2    & 8           &    18  \\ \cline{2-5} 
residual  & \multicolumn{1}{|l|}{\texttt{LM-(global and local)-V2}  }           &   4    &     9       &   15   \\ \cline{2-5} 
  & \multicolumn{1}{|l|}{\texttt{LM-(global)}}  &  18     &      9 &    1 \\ \hline \hline
 nonzero &\multicolumn{1}{|l|}{\texttt{LM-(global and local)-V1}  }           &   7    & 7           &    5  \\ \cline{2-5} 
 residual  & \multicolumn{1}{|l|}{\texttt{LM-(global and local)-V2}  }          &      10 &      8      &  1    \\ \cline{2-5} 
 & \multicolumn{1}{|l|}{\texttt{LM-(global)}}  &  12     &      5 &    2  \\ \hline
\end{tabular}
\end{table}
The estimation of the order of convergence (see Table \ref{table:eoc:P1}) shows the good local behavior of the \texttt{LM-(global and local)} variants compared to \texttt{LM-(global)}. 
In fact, \texttt{LM-(global and local)} variants converged quadratically or superlinearly on $38$ (\texttt{v1}) and $33$ (\texttt{v2}) problems respectively, while \texttt{LM-(global)} showed quadratic or superlinear convergence for only 17 problems.

\section{Conclusions}
\label{sec:conclusion}
In this paper, we presented and analyzed a novel LM method for solving nonlinear least-squares problems. 
%Without the use of ancillary procedures to enforce globalization,
We were able to formulate a globally convergent LM method with strong worst-case iteration complexity bounds. 
The proposed method is locally convergent at quadratic rate for zero residual problems and at a linear rate for small residuals. 
Preliminary numerical results confirmed the theoretical behavior. 
Future research can include problems with constraints as well as those with noisy data.

\section*{Acknowledgements}
\small{
We would like to thank Cl\'{e}ment Royer and the referees for their careful readings and corrections that helped us to improve our manuscript significantly. Support for Vyacheslav Kungurtsev was provided by the OP VVV project
CZ.02.1.01/0.0/0.0/16\_019/0000765 ``Research Center for Informatics''.}
\bibliographystyle{spmpsci_unsrt}
\addcontentsline{toc}{chapter}{Bibliography}
%\bibliography{thebib

%}

\end{document}